\documentclass[a4paper,11pt]{amsart}

\usepackage{amsmath,amssymb,amsthm,mathtools,hyperref}
\allowdisplaybreaks

\numberwithin{equation}{section}
\theoremstyle{plain}
\newtheorem{thm}{Theorem}[section]
\newtheorem{prop}[thm]{Proposition}
\newtheorem{lem}[thm]{Lemma}

\newtheorem{rem}[thm]{Remark}

\makeatletter

\allowdisplaybreaks[3]

\def\tri{\triangle}

\DeclareMathOperator{\Tr}{Tr}
\DeclareMathOperator{\vol}{Vol}

\makeatother

\title[Volume gap between the minimal submanifold and the unit sphere]{Volume gap between the minimal submanifold and the unit sphere}

\author[W.R. Ding]{Weiran Ding$^{1}$}
\address{$^{1}$School of Mathematical Sciences, Laboratory of Mathematics and Complex Systems, Beijing Normal University, Beijing 100875, P. R. CHINA.}
\email{dingwr0806@mail.bnu.edu.cn}

\author[J.Q. Ge]{Jianquan Ge$^{2}$}
\address{$^{2}$School of Mathematical Sciences, Laboratory of Mathematics and Complex Systems, Beijing Normal University, Beijing 100875, P. R. CHINA.}
\email{jqge@bnu.edu.cn}

\author[F.G. Li]{Fagui Li$^{3,*}$}
\address{$^{3,*}$Greater Bay Area Innovation Research Institute, Beijing Institute of Technology, Zhuhai, Guangdong 519088, P. R. CHINA.}
\email{lifagui@bitzh.edu.cn}

\date{}

\thanks{J. Q. Ge is partially supported by NSFC (No. 12171037) and the Fundamental Research Funds for the Central Universities.}
\thanks{F. G. Li is partially supported by NSFC (No. 12171037, 12271040) and China Postdoctoral Science Foundation (No. 2022M720261).}
\begin{document}
\maketitle

\begin{abstract}
	In 1984, Cheng-Li-Yau [Amer. J. Math. \textbf{106} (1984), 1033--1065] proved that for a compact minimally immersed submanifold $M^n$ of $\mathbb{S}^{n+\ell}$ there exists a positive constant $B_n$ depending only on $n$ such that $$\vol(M)>\left(1+\frac{2\ell-1}{B_n}\right)\vol(\mathbb{S}^n).$$ In this paper, following the method of Cheng-Li-Yau, we first modify the coefficients in the constant $B_n$ to improve the volume gap. Further, we also enlarge our gap by applying an estimate of Cheng-Yang [Math. Ann., \textbf{331} (2005), 445--460] for eigenvalues of Laplacian.
\end{abstract}

\section{Introduction}
Let $M^n$ be a $n$-dimensional minimally immersed submanifold of $\mathbb{S}^{n+\ell}$, where $\ell\ge1$. In 1984, Cheng-Li-Yau \cite{Cheng Li Yau 1984 Heat equations} proved the following volume gap theorem. This theorem means that there is a volume gap for minimal submanifolds in spheres.
\begin{thm}[Cheng-Li-Yau \cite{Cheng Li Yau 1984 Heat equations}]\label{Cheng Li Yau 1984 Heat equations}
	Let $M$ be a compact minimally immersed submanifold of $\mathbb{S}^{n+\ell}$. Suppose the immersion is of maximal dimension, then there exists a positive constant
	\begin{equation*}
		B_n<2n+3+2\exp(2nC_n)
	\end{equation*}
	such that
	\begin{equation*}
		\vol(M)>\left(1+\frac{2\ell-1}{B_n}\right)\vol(\mathbb{S}^n)
	\end{equation*}
	where
	\begin{equation*}
		C_n=\frac{n^{n/2}e\Gamma(n/2,1)}{2},\quad\Gamma\left(\frac{n}{2},1\right)=\int_{1}^\infty e^{-t}t^{\frac{n}{2}-1}~dt
	\end{equation*}
	and $\ell$ is the co-dimension of the immersion in $\mathbb{S}^{n+\ell}$.
\end{thm}
In 1992, Yau put forward the following famous conjecture (also called the Solomon-Yau conjecture \cite{Ge19,IW15}) in his Problem Section \cite{Yau  1992} about the second smallest volume: \emph{the volume of one of the minimal Clifford torus $M_{k,n-k}=S^k(\sqrt{\frac{k}{n}})\times S^{n-k}(\sqrt{\frac{n-k}{n}}),1\le k\le n-1$ gave the lowest value of volume among all non-totally geodesic closed minimal hypersurfaces of $\mathbb{S}^{n+1}$.} For $n=2$, the Yau conjecture is true, due to Li-Yau \cite{Li and Yau 1982}, Calabi \cite{Calabi 1967} and Marques-Neves \cite{Marques and Neves 2014}. For $n\geq 3$, there are many partial results. The Yau conjecture was verified for minmal rotational hypersurfaces by Perdomo-Wei \cite{Perdomo and Wei 2015}, Cheng-Wei-Zeng \cite{Cheng Qing Ming Guoxin Wei and Yuting Zeng 2019} and Cheng-Lai-Wei \cite{Cheng Qing Ming  Junqi Lai and Guoxin Wei 2024}. Ilmanen-White \cite{IW15} proved the Yau conjecture in the asymptotic sense under some conditions. Viana \cite{Celso Viana 2023} confirmed the Yau Conjecture in the class of minimal hypersurfaces separating $\mathbb{S}^n$ in two connected regions which are both antipodal invariant. Recently, Ge-Li \cite{Ge Li 2022 Solomon-Yau conj} and Nguyen \cite{Nguyen 2023} proved the Yau conjecture for  nonembedded hypersurfaces, independently. Under some curvature conditions, Ge-Li \cite{Ge Li 2022 Solomon-Yau conj} also got some better results of the gap for embedded hypersurfaces. In this paper, without any assumptions of the immersion and the curvature, we prove the following results:
\begin{thm}\label{1st}
	Let $M$ be a compact minimally immersed submanifold of $\mathbb{S}^{n+\ell}$. Suppose the immersion is of maximal dimension, then there exists a positive constant
	\begin{equation*}
		B_{n,\alpha}<\alpha n+\alpha+1+\alpha\exp(\alpha nC_n)
	\end{equation*}
	with $\alpha=1.43$ such that
	\begin{equation*}
		\vol(M)>\left(1+\frac{\alpha\ell-1}{B_{n,\alpha}}\right)\vol(\mathbb{S}^n)>\left(1+\frac{1.65(2\ell-1)}{B_n}\right)\vol(\mathbb{S}^n).
	\end{equation*}
\end{thm}
\begin{rem}
	It should be noted that our gap is much larger than the original volume gap given by Cheng-Li-Yau \cite{Cheng Li Yau 1984 Heat equations}. Moreover, the improvement becomes better as $n$ and $\ell$ increase.
\end{rem}
Applying the estimate of the Laplacian eigenvalues by Cheng-Yang \cite{ChY}, we improved our gap to the following:
\begin{thm}\label{2nd}
	Let $M$ be a compact minimally immersed submanifold of $\mathbb{S}^{n+\ell}$. Suppose the immersion is of maximal dimension, then there exist positive constants
	\begin{equation*}
		B_{n,\alpha}<\alpha n+\alpha+1+\alpha\exp(\alpha nC_n)
	\end{equation*}
	and constant $E_{n,\ell}=\alpha nC_n-\alpha nC_n(n+4)(n+2\ell)^{\frac{2}{n}}4^{\frac{1}{n}}$ with $\alpha=1.43$ such that
	\begin{equation*}
		\begin{aligned}
			\vol(M)&>\left(1+\frac{\alpha\ell-1+\alpha(n+\ell+2)e^{E_{n,\ell}}}{B_{n,\alpha}}\right)\vol(\mathbb{S}^n)\\
			&>\left(1+\frac{1.65(2\ell-1)}{B_n}+\frac{1.65(n+\ell+2)e^{E_{n,\ell}}}{\ell}\frac{2\ell-1}{B_n}\right)\vol(\mathbb{S}^n).
		\end{aligned}
	\end{equation*}
\end{thm}

\section{The Heat Kernel}
First we introduce the basic definitions and propositions of heat kernel and summarize the comparison results by Cheng-Li-Yau \cite{Cheng Li Yau 1984 Heat equations}. Let $M^n$ be a $n$-dimensional minimally immersed submanifold of $\mathbb{S}^{n+\ell}$. For $p\in M$, let $r_p(x)$ be the distance function on $\mathbb{S}^{n+\ell}$ and we denote the restriction of $r_p$ to $M$ as the extrinsic distance function on $M$. For any $a>0$, we define the extrinsic ball centered at $p$ with radius $a$ by
\begin{equation*}
	D_p(a)=B_p(a)\cap M
\end{equation*}
where $B_p(a)=\{x\in\mathbb{S}^{n+\ell}~|~r_p(x)\le a\}$. Let $D\subset M$ be a compact domain. For any $p\in D$, the extrinsic outer radius at $p$ is defined by $a=\sup_{z\in D}r_p(z)$. We denote the heat kernels for the Dirichlet and the Neumann boundary conditions by $H(x,y,t)$ and $K(x,y,t)$ respectively. The following basic properties are well known:
\begin{enumerate}
	\item[(i)] for all $x,y\in D$ and $t\in[0,\infty)$, we have $\square_yH(x,y,t)=\square_yK(x,y,t)=0$;
	\item[(ii)] for all $x\in D$, we have $H(x,y,0)=K(x,y,0)=\delta_x$;
	\item[(iii)] for all $z\in\partial D$, we have $H(x,z,t)=0$ and $$\dfrac{\partial}{\partial\nu_z}K(x,z,t)=0,$$ where $\partial/\partial\nu_z$ stands for the differentiation in the $z$ variable in the outward normal direction to $\partial D$.
\end{enumerate}
To introduce the comparison results of the minimal submanifolds, Cheng-Li-Yau \cite{Cheng Li Yau 1984 Heat equations} first presented the following proposition.
\begin{prop}[Proposition 1 in \cite{Cheng Li Yau 1984 Heat equations}]\label{prop1}
	Let $M$ be a compact manifold with boundary $\partial M$. Suppose $p\in M$, and if $G(p,y,t),\bar{G}(p,y,t):M\times M\times[0,\infty)\to\mathbb{R}$ are $C^2$ functions with the properties that:
	\begin{enumerate}
		\item[(i)] for all $y\in M$ and $t\in[0,\infty)$, we have $G(p,y,t)\ge0$;
		\item[(ii)] for all $y\in M$, we have $G(p,y,0)=\bar{G}(p,y,0)=\delta_p$;
		\item[(iii)] for all $y\in M$ and $t\in[0,\infty)$, we have $\square_yG(p,y,t)=0$ and $\square_y\bar{G}(p,y,t)\le0$;
		\item[(iv)] for all $z\in\partial M$ and $t\in[0,\infty)$, we have $G(p,z,t)=0$ and $\bar{G}(p,z,t)\ge0$ or $$\dfrac{\partial}{\partial\nu_z}G(p,z,t)=\dfrac{\partial}{\partial\nu_z}\bar{G}(p,z,t)=0.$$
	\end{enumerate}
	Then $$G(p,y,t)\le\bar{G}(p,y,t)$$ for all $y\in M$ and $t\in[0,\infty)$.
\end{prop}
The proof of Proposition \ref{prop1} can be found in \cite{CheeYau} and \cite{PLi}. Based on Proposition \ref{prop1}, Cheng-Li-Yau \cite{Cheng Li Yau 1984 Heat equations} proved the following three comparison results of the minimal submanifolds immersed in the unit sphere..
\begin{prop}[Theorem 3 in \cite{Cheng Li Yau 1984 Heat equations}]\label{thm3}
	Let $M^n$ be a $n$-dimensional minimally immersed submanifold of $\mathbb{S}^{n+\ell}$. Suppose that $D$ is a compact domain in $M$. For $p\in D$ if the outer radius at $p$ is not greater than $\frac{\pi}{2}$, then
	\begin{equation*}
		H(p,y,t)\le\bar{H}_a(r_p(y),t)
	\end{equation*}
	for all $y\in D$ and $t\in[0,\infty)$, where $\bar{H}_a(r_p(y),t)$ is the heat kernel with Dirichlet boundary condition on the ball of radius $a$ in $\mathbb{S}^{n}$ centered at the northpole.
\end{prop}
\begin{prop}[Theorem 4 in \cite{Cheng Li Yau 1984 Heat equations}]\label{thm4}
	Let $M^n$ be a $n$-dimensional minimally immersed submanifold of $\mathbb{S}^{n+\ell}$. Let $D_p(a)$ be any extrinsic ball of $M$ with radius $0<a<\pi$. If $\bar{K}_a(r_p(y),t)$ is the heat kernel with Neumann boundary condition on the ball of radius $a$ in $\mathbb{S}^n$, then
	\begin{equation*}
		K(p,y,t)\le\bar{K}_a(r_p(y),t)
	\end{equation*}
	for all $y\in D_p(a)$ and $t\in[0,\infty)$.
\end{prop}
\begin{prop}[Theorem 5 in \cite{Cheng Li Yau 1984 Heat equations}]\label{thm5}
	Let $M^n$ be a $n$-dimensional compact and minimally immersed submanifold into $\mathbb{S}^{n+\ell}$. If we denote the heat kernel on $M$ (without boundary condition) by $K(x,y,t)$ and the heat kernel on $\mathbb{S}^n$ by $\bar{K}(x,y,t)$, then
	\begin{equation*}
		K(p,y,t)\le\bar{K}(r_p(y),t)
	\end{equation*}
	for all $p,y\in M$ and $t\in[0,\infty)$.
\end{prop}
Using Proposition \ref{thm3}, Proposition \ref{thm4} and Proposition \ref{thm5}, Cheng-Li-Yau \cite{Cheng Li Yau 1984 Heat equations} proved some interesting applications. The most interesting application comes from Proposition \ref{thm5}. First it is necessary to estimate the trace of the heat kernel on $\mathbb{S}^n$.
\begin{prop}[Lemma 3 in \cite{Cheng Li Yau 1984 Heat equations}]\label{lem3}
	Let $\bar{K}(x,y,t)$ be the heat kernel on $\mathbb{S}^n$, then there exists a constant $C_n$ depending only on $n$ with $$C_n=\frac{n^{n/2}e\Gamma(n/2,1)}{2}$$ such that
	\begin{equation*}
		\Tr\bar{K}(t)=\int_{\mathbb{S}^n}\bar{K}(x,x,t)~dx\le1+(n+1)e^{-nt}+C_nt^{-1}e^{-nt}
	\end{equation*}
	for all $t\ge1$.
\end{prop}
\begin{rem}
	Using a result of Takahashi \cite{Takahashi}, an immersion of $M^n$ into $\mathbb{S}^{n+\ell}$ is minimal if and only if $\tri\phi=-n\phi$ for any coordinate function $\phi$ in $\mathbb{R}^{n+\ell+1}$. Moreover, if $M^n$ is compact, then $n=\lambda_k$ is the $k$-th eigenvalue of $M$. In addition, if the immersion is of maximal dimension, i.e., the image of $M^n$ is not contained in any hyperplane of $\mathbb{R}^{n+\ell+1}$, then the multiplicity of $n=\lambda_k$ is at least $n+\ell+1$.
\end{rem}
By Proposition \ref{thm5} and Proposition \ref{lem3}, Cheng-Li-Yau \cite{Cheng Li Yau 1984 Heat equations} proved the following useful inequality.
\begin{prop}[Theorem 6 in \cite{Cheng Li Yau 1984 Heat equations}]\label{thm6}
	Let $M^n\rightarrow\mathbb{S}^{n+\ell}$ be a minimal immersion of the compact manifold $M$ of maximal dimension. Suppose that the spectrum of $M$ is ordered by magnitude. If $n=\lambda_k$, then
	\begin{equation*}
		k\le\frac{\vol(M)}{\vol(\mathbb{S}^n)}(e^t+n+1+nC_nt^{-1})-e^t,
	\end{equation*}
	for any $t\ge0$.
\end{prop}
Finally, 
 Cheng-Li-Yau \cite{Cheng Li Yau 1984 Heat equations} completed the proof of Theorem \ref{Cheng Li Yau 1984 Heat equations} by Proposition \ref{thm6},.

\section{Proof of Theorem \ref{1st}}
First we consider a function $f_1:[0,s]\to\mathbb{R}$ with
\begin{equation*}
	f_1(\alpha)=\frac{\alpha\ell-1}{\alpha n+\alpha+1+\alpha\exp(\alpha nC_n)},
\end{equation*}
where $s>2$ is any positive real number. Since $f_1$ is continuous on $[0,s]$, then $f_1$ can attain its maximal value and minimal value on $[0,s]$. Let $\beta$ be the point where $f_1$ attains its maximal value.
\begin{lem}\label{leml}
	Under the foregoing notations and assumptions, the larger $\ell$ is, the smaller $\beta$ is.
\end{lem}
\begin{proof}
	First, we have
	\begin{equation*}
		f_1'(\alpha)=\frac{-\ell(\alpha^2nC_n\exp(\alpha nC_n)-1)+(n+1+(1+\alpha nC_n)\exp(\alpha nC_n))}{(\alpha n+\alpha+1+\alpha\exp(\alpha nC_n))^2}.
	\end{equation*}
	Since $\beta$ is the point where $f_1$ attains its maximal value, then $f_1'(\beta)=0$. Hence
	\begin{equation*}
		\ell=\frac{n+1+(1+\beta nC_n)\exp(\beta nC_n)}{\beta^2nC_n\exp(\beta nC_n)-1}.
	\end{equation*}
	Let
	\begin{equation*}
		g(\beta)=\frac{n+1+(1+\beta nC_n)\exp(\beta nC_n)}{\beta^2nC_n\exp(\beta nC_n)-1}.
	\end{equation*}
	Then
	\begin{equation*}
		g'(\beta)=\text{I}+\text{II}
	\end{equation*}
	where
	\begin{equation*}
		\begin{aligned}
			\text{I}&=\frac{nC_ne^{\beta nC_n}(2+\beta nC_n)(\beta^2nC_ne^{\beta nC_n}-1)}{(\beta^2nC_ne^{\beta nC_n}-1)^2},\text{ and}\\
			\text{II}&=\frac{-nC_ne^{\beta nC_n}(2\beta+\beta^2nC_n)[n+1+(1+\beta nC_n)e^{\beta nC_n}]}{(\beta^2nC_ne^{\beta nC_n}-1)^2}.
		\end{aligned}
	\end{equation*}
	Then
	\begin{equation*}
		\begin{aligned}
			g'(\beta)&=\frac{-nC_ne^{2\beta nC_n}(2\beta+\beta^2nC_n)-nC_ne^{\beta nC_n}[\beta nC_n(1+\beta(n+1))+2(\beta n+\beta+1)]}{(\beta^2nC_ne^{\beta nC_n}-1)^2}\\
			&<0.
		\end{aligned}
	\end{equation*}
	Therefore, the larger $\ell$ is, the smaller $\beta$ is.
\end{proof}
By Lemma \ref{leml}, it suffice to consider the smallest $\ell$, i.e., $\ell=1$. Hence, it suffices to consider
\begin{equation*}
	f_2(\beta)=\frac{\beta-1}{\beta n+\beta+1+\beta\exp(\beta nC_n)}.
\end{equation*}
Let $\gamma_n$ be the point where $f_2$ attains its maximal value for $n\ge2$.
\begin{lem}\label{lemn}
	Under the foregoing notations and assumptions, we have $\gamma_2>1.3$ and $1<\gamma_n\le1.3$ for $n\ge3$.
\end{lem}
\begin{proof}
	First, we have
	\begin{equation*}
		f_2'(\beta)=\frac{(\beta n+\beta+1+\beta\exp(\beta nC_n))-(\beta-1)(n+1+\exp(\beta nC_n)+\beta nC_n\exp(\beta nC_n))}{(\beta n+\beta+1+\beta\exp(\beta nC_n))^2}.
	\end{equation*}
	Then
	\begin{equation}\label{lemeq}
		n+2=(\gamma_n^2nC_n-\gamma_nnC_n-1)\exp(\gamma_nnC_n).
	\end{equation}
	From this equation, if $\gamma\le1$, then the left-hand-side of \eqref{lemeq} is positive while the right-hand-side of \eqref{lemeq} is negative, which gives the contradiction. Hence we obtain that $\gamma>1$ for every $n\ge2$. From now on, we consider the equation
	\begin{equation}\label{lemeqq}
		\frac{n+2}{\exp(\gamma_nnC_n)}=\gamma_n^2nC_n-\gamma_nnC_n-1,
	\end{equation}
	The left-hand-side and the right-hand-side of \eqref{lemeqq} with respect to $n$ will be denoted by $\text{LHS}|_n$ and $\text{RHS}|_n$, respectively. Next, we will   discuss the following three cases.
	\begin{itemize}
		\item[(i)]For $n=2$, the equation \eqref{lemeqq} becomes
			\begin{equation*}
				\frac{4}{\exp(2\gamma_2)}=2\gamma_2^2-2\gamma_2-1.
			\end{equation*}
			Now suppose that $1<\gamma_2\le1.3$. Then $\text{LHS}|_2\ge\frac{4}{e^{2.6}}>0.29$, while $\text{RHS}|_2<0$ when $1<\gamma_2\le1.3$, which gives the contradiction.
		\item[(ii)] For $n=3$, the equation \eqref{lemeqq} becomes
			\begin{equation*}
				\frac{5}{\exp(3C_3\gamma_3)}=3C_3\gamma_3^2-3C_3\gamma_3-1.
			\end{equation*}
			It is easy to calculate that $C_3\approx 3.58$. Actually we only need $C_3>\frac{10}{3}$. Suppose that $\gamma_3>1.3$, then $\text{LHS}|_3<\frac{5}{e^{13}}\ll0.1$. To estimate $\text{RHS}|_3$, let $$\varphi_3(\gamma_3)=3C_3\gamma_3^2-3C_3\gamma_3-1.$$ Let $\varphi_3(\tilde{\gamma}_3)=0$. Combining $\tilde{\gamma}_3>1$, we obtain
			\begin{equation*}
				\tilde{\gamma}_3=\frac{3C_3+\sqrt{(3C_3)^2+4(3C_3)}}{2(3C_3)}=\frac{1+\sqrt{1+\frac{4}{3C_3}}}{2}<1.1,
			\end{equation*}
			which implies that $\varphi_3(\gamma_3)$ is monotone increasing in $\gamma_3\in(1.1,+\infty)$. Hence, if $\gamma_3>1.3$, then we have
			\begin{equation*}
				\text{RHS}|_3=\varphi_3(\gamma_3)>\varphi_3(1.3)=1.17C_3-1>2,
			\end{equation*}
			which gives the contradiction. Therefore, $1<\gamma_3\le1.3$.
		\item[(iii)] For $n\ge4$,  suppose that $\gamma_n>1.3$ and we observe that
			\begin{equation*}
				\begin{aligned}
					\text{RHS}|_{n\ge4}&=\gamma_n^2nC_n-\gamma_nnC_n-1\\
					&=nC_n\gamma_n(\gamma_n-1)-1\\
                   &\ge4C_4\gamma_4(\gamma_4-1)-1\\
					&\ge4\times16\times1.3\times0.3-1\\
					&>20
				\end{aligned}
			\end{equation*}
			by $C_4=16$ and the monotonicity of $C_n$ for $n\ge4$. On the other hand, we have
			\begin{equation*}
				\text{LHS}|_{n\ge4}=\frac{n+2}{\exp(\gamma_nnC_n)}<\frac{n+2}{\exp(20n)}\eqqcolon\psi(n).
			\end{equation*}
			Since $\psi(n)$ is decreasing for $n\ge4$, then $\text{LHS}|_{n\ge4}<\psi(4)\ll1$, which gives the contradiction.
	\end{itemize}
	This completes the proof by (i)-(iii).
\end{proof}
Now we can give the proof of Theorem \ref{1st}.
\begin{proof}[Proof of Theorem \ref{1st}]
	Choosing $t=\alpha nC_n$ in Proposition \ref{thm6}, we obtain
	\begin{equation*}
		k\le\frac{\vol(M)}{\vol(\mathbb{S}^n)}\left(e^{\alpha nC_n}+n+1+nC_n\frac{1}{\alpha nC_n}\right)-e^{\alpha nC_n}.
	\end{equation*}
	Since the immersion is maximal dimension, the multiplicity of $n$ is at least $n+\ell+1$, hence $k\ge n+\ell+1$. Then
	\begin{equation*}
		n+\ell+1\le\frac{\vol(M)}{\vol(\mathbb{S}^n)}\left(e^{\alpha nC_n}+n+1+\frac{1}{\alpha}\right)-e^{\alpha nC_n},
	\end{equation*}
	which implies that
	\begin{equation*}
		\vol(M)\ge\left(1+\frac{\alpha\ell-1}{\alpha n+\alpha+1+\alpha\exp(\alpha nC_n)}\right)\vol(\mathbb{S}^n).
	\end{equation*}
	By Lemma \ref{leml} and Lemma \ref{lemn}, it suffice to consider the case $(n,\ell)=(2,1)$, i.e.,
	\begin{equation*}
		f(\alpha)=\frac{\alpha-1}{3\alpha+1+\alpha\exp(2\alpha)}.
	\end{equation*}
	It remains to find the maximum point of the above function. To do this, we first have
	\begin{equation*}
		f'(\alpha)=\frac{4+(1+2\alpha-2\alpha^2)\exp(2\alpha)}{(3\alpha+1+\alpha\exp(2\alpha))^2}.
	\end{equation*}
	Let
	\begin{equation*}
		h(\alpha)=4+(1+2\alpha-2\alpha^2)\exp(2\alpha).
	\end{equation*}
	It is easy to calculate that $h(1.42)>0$ and $h(1.44)<0$. Then the zero point lies between $(1.42,1.44)$. For simplicity, we can choose $\alpha=1.43$ to obtain a better gap than \cite{Cheng Li Yau 1984 Heat equations}. In conclusion, there exists a positive constant
	\begin{equation*}
		B_{n,\alpha}<\alpha n+\alpha+1+\alpha\exp(\alpha nC_n)
	\end{equation*}
	with $\alpha=1.43$ such that
	\begin{equation*}
		\vol(M)>\left(1+\frac{\alpha\ell-1}{B_{n,\alpha}}\right)\vol(\mathbb{S}^n),
	\end{equation*}
	which proves the first inequality. To prove the second inequality, we have the following calculations:
	\begin{equation*}
		\begin{aligned}
			\left(\frac{\alpha\ell-1}{B_{n,\alpha}}\right)\bigg/\left(\frac{2\ell-1}{B_n}\right)&=\frac{1.43\ell-1}{2\ell-1}\frac{2n+3+2\exp(2nC_n)}{1.43n+2.43+1.43\exp(1.43nC_n)}.
		\end{aligned}
	\end{equation*}
	Note that $$\frac{1.43\ell-1}{2\ell-1}\text{\quad and \quad}\frac{2n+3+2\exp(2nC_n)}{1.43n+2.43+1.43\exp(1.43nC_n)}$$ are increasing with respect to $\ell$ and $n$ respectively. Then
	\begin{equation*}
		\begin{aligned}
			\left(\frac{\alpha\ell-1}{B_{n,\alpha}}\right)\bigg/\left(\frac{2\ell-1}{B_n}\right)&=\frac{1.43\ell-1}{2\ell-1}\frac{2n+3+2\exp(2nC_n)}{1.43n+2.43+1.43\exp(1.43nC_n)}\\
			&\ge\frac{1.43-1}{2-1}\frac{7+2e^4}{5.29+1.43e^{2.86}}\\
			&>1.65.
		\end{aligned}
	\end{equation*}
	Therefore
	\begin{equation*}
		\left(1+\frac{\alpha\ell-1}{B_{n,\alpha}}\right)\vol(\mathbb{S}^n)>\left(1+\frac{1.65(2\ell-1)}{B_n}\right)\vol(\mathbb{S}^n),
	\end{equation*}
	which proves the second inequality.
\end{proof}

\section{Proof of Theorem \ref{2nd}}
It is well known that the eigenvalue problem of Laplacian has a real and purely discrete spectrum
\begin{equation*}
	0<\lambda_1<\lambda_2\le\lambda_3\le\cdots\rightarrow\infty.
\end{equation*}
Here each eigenvalue is repeated from its multiplicity. Cheng-Yang \cite{ChY} proved the following estimate of the eigenvalues.
\begin{thm}[Cheng-Yang \cite{ChY}]\label{ChY}
	Let $M^n$ be an $n$-dimensional compact minimal submanifold without boundary in a unit sphere $\mathbb{S}^N(1)$. If $\lambda_i,~i=1,2,\cdots,$ is the $i$-th eigenvalue of $\tri u=-\lambda u$ on $M^n$, then
	\begin{equation*}
		\lambda_{k+1}\le D_nk^{\frac{2}{n}}\lambda_1,
	\end{equation*}
	where $D_n\le n+4$ is a positive constant only depended on $n$.
\end{thm}
We first prove the following result by using Theorem \ref{ChY}.
\begin{thm}\label{2ndd}
	Let $M$ be a compact minimally immersed submanifold of $\mathbb{S}^{n+\ell}$. Suppose the immersion is of maximal dimension, then there exist positive constants
	\begin{equation*}
		B_{n,\alpha}<\alpha n+\alpha+1+\alpha\exp(\alpha nC_n)
	\end{equation*}
	and constant $E_{n,\ell}=\alpha nC_n-\alpha nC_n(n+4)(n+2\ell)^{\frac{2}{n}}4^{\frac{1}{n}}$ with $\alpha=1.43$ such that
	\begin{enumerate}
		\item[(i)] if $n+\ell+1\le k\le n+2\ell$, then
			\begin{equation*}
				\vol(M)>\left(1+\frac{\alpha\ell-1+\alpha(n+\ell+2)e^{E_{n,\ell}}}{B_{n,\alpha}}\right)\vol(\mathbb{S}^n)
			\end{equation*}
		\item[(ii)] if $k\ge n+2\ell+1$, then
			\begin{equation*}
				\vol(M)>\left(1+\frac{2\alpha\ell-1}{B_{n,\alpha}}\right)\vol(\mathbb{S}^n).
			\end{equation*}
	\end{enumerate}
\end{thm}
\begin{proof}
	By Proposition \ref{thm5}, the heat kernel
	\begin{equation*}
		K(x,x,t)=\sum_{i=0}^{\infty}e^{-\lambda_i t\phi_i^2(x)}
	\end{equation*}
	satisfies
	\begin{equation*}
		\sum_{i=0}^{\infty}e^{-\lambda_i t\phi_i^2(x)}\le\bar{K}(x,x,t).
	\end{equation*}
	However it is easy to see that $\bar{K}(x,x,t)$ is a constant function since $\mathbb{S}^n$ is a homogeneous manifold, therefore
	\begin{equation*}
		\int_{\mathbb{S}^n}\bar{K}(x,x,t)=\vol(\mathbb{S}^n)\bar{K}(x,x,t).
	\end{equation*}
	Together with Proposition \ref{lem3}, this gives
	\begin{equation*}
		\begin{aligned}
			\sum_{i=0}^{\infty}e^{-\lambda_i t}&\le\int_MK(x,x,t)~dx\\
			&\le\int_M\bar{K}(x,x,t)~dx\\
			&=V(M)\bar{K}(x,x,t)\\
			&=\frac{V(M)}{V(\mathbb{S}^n)}\int_{\mathbb{S}^n}\bar{K}(x,x,t)~dx\\
			&\le\frac{V(M)}{V(\mathbb{S}^n)}\left(1+(n+1)e^{-nt}+C_nt^{-1}e^{-nt}\right).
		\end{aligned}
	\end{equation*}
	Rescaling $t$ by $\frac{t}{\lambda}$ and noting that $\lambda_k=n$, we obtain
	\begin{equation*}
		\sum_{i=0}^{\infty}e^{-\frac{\lambda_i}{\lambda_k}t}\le\frac{\vol(M)}{\vol(\mathbb{S}^n)}\left(1+(n+1)e^{-t}+nC_nt^{-1}e^{-t}\right)
	\end{equation*}
	for $t\ge0$. Therefore, we obtain
	\begin{equation}\label{all}
		1+\sum_{i=1}^{k}e^{-\frac{\lambda_i}{\lambda_k}t}+\sum_{i=k+1}^{\infty}e^{-\frac{\lambda_i}{\lambda_k}t}\le\frac{\vol(M)}{\vol(\mathbb{S}^n)}\left(1+(n+1)e^{-t}+nC_nt^{-1}e^{-t}\right)
	\end{equation}
	for all $t\ge0$. On one hand, since $\lambda_i\le\lambda_k$ for $i\le k$, we have
	\begin{equation}\label{from0tok}
		1+\sum_{i=1}^{k}e^{-\frac{\lambda_i}{\lambda_k}t}\ge1+ke^{-t}.
	\end{equation}
	On the other hand, we have
	\begin{equation*}
		\begin{aligned}
			\sum_{i=k+1}^{\infty}e^{-\frac{\lambda_i}{\lambda_k}t}&\ge\sum_{i=k+1}^{2k+1}e^{-\frac{\lambda_i}{\lambda_k}t}
			\ge(k+1)\exp\left(-\frac{\lambda_{2k+1}}{\lambda_k}t\right)\\
			&\ge(k+1)\exp\left(-\frac{D_n}{n}(2k)^{\frac{2}{n}}\lambda_1t\right).
		\end{aligned}
	\end{equation*}
	Since the immersion is maximal dimension, the multiplicity of $n$ is at least $n+\ell+1$, hence $k\ge n+\ell+1$. Here, we split the condition $k\ge n+\ell+1$ into two cases. First, if $n+\ell+1\le k\le n+2\ell$, then we have
	\begin{equation}\label{fromk+1toinfty}
		\sum_{i=k+1}^{\infty}e^{-\frac{\lambda_i}{\lambda_k}t}\ge(k+1)\exp\left(-\frac{D_n}{n}(n+2\ell)^{\frac{2}{n}}\lambda_14^{\frac{1}{n}}t\right).
	\end{equation}
	Combining \eqref{from0tok} and \eqref{fromk+1toinfty}, we obtain
	\begin{equation*}
		\frac{\vol(M)}{\vol(\mathbb{S}^n)}\left(1+(n+1)e^{-t}+nC_nt^{-1}e^{-t}\right)\ge1+ke^{-t}+(k+1)\exp\left(-\frac{D_n}{n}(n+2\ell)^{\frac{2}{n}}\lambda_14^{\frac{1}{n}}t\right).
	\end{equation*}
	Let $t=\alpha nC_n$ with $\alpha=1.43$, then we obtain
	\begin{equation*}
		\begin{aligned}
			\frac{\vol(M)}{\vol(\mathbb{S}^n)}&\ge\frac{1+ke^{-t}+(k+1)\exp\left(-\frac{D_n}{n}(n+2\ell)^{\frac{2}{n}}\lambda_14^{\frac{1}{n}}t\right)}{1+(n+1)e^{-t}+nC_n\frac{1}{t}e^{-t}}\\
			&=\frac{e^t+k+(k+1)\exp\left(t-\frac{D_n}{n}(n+2\ell)^{\frac{2}{n}}\lambda_14^{\frac{1}{n}}t\right)}{e^t+(n+1)+nC_n\frac{1}{t}}\\
			&=\frac{e^{\alpha nC_n}+k+(k+1)\exp\left(\alpha nC_n-\alpha C_nD_n(n+2\ell)^{\frac{2}{n}}\lambda_14^{\frac{1}{n}}\right)}{e^{\alpha nC_n}+(n+1)+\frac{1}{\alpha}}\\
			&\ge\frac{e^{\alpha nC_n}+n+\ell+1+(n+\ell+2)\exp\left(\alpha nC_n-\alpha C_nD_n(n+2\ell)^{\frac{2}{n}}\lambda_14^{\frac{1}{n}}\right)}{e^{\alpha nC_n}+(n+1)+\frac{1}{\alpha}}\\
			&=1+\frac{\ell-\frac{1}{\alpha}+(n+\ell+2)\exp\left(\alpha nC_n-\alpha C_nD_n(n+2\ell)^{\frac{2}{n}}\lambda_14^{\frac{1}{n}}\right)}{e^{\alpha nC_n}+(n+1)+\frac{1}{\alpha}}.
		\end{aligned}
	\end{equation*}
	Since $\lambda_1<n$, then we obtain
	\begin{equation*}
		\begin{aligned}
			\frac{\vol(M)}{\vol(\mathbb{S}^n)}&>1+\frac{\ell-\frac{1}{\alpha}+(n+\ell+2)\exp\left(\alpha nC_n-\alpha nC_nD_n(n+2\ell)^{\frac{2}{n}}4^{\frac{1}{n}}\right)}{e^{\alpha nC_n}+(n+1)+\frac{1}{\alpha}}\\
			&=1+\frac{\alpha\ell-1+\alpha(n+\ell+2)e^{E_{n,\ell}}}{\alpha n+\alpha+1+\alpha e^{\alpha nC_n}},
		\end{aligned}
	\end{equation*}
	where we choose $D_n=n+4$ for simplicity and $E_{n,\ell}=\alpha nC_n-\alpha nC_n(n+4)(n+2\ell)^{\frac{2}{n}}4^{\frac{1}{n}}$. Hence,
	\begin{equation}\label{1intervel}
		\vol(M)\ge\left(1+\frac{\alpha\ell-1+\alpha(n+\ell+2)e^{E_{n,\ell}}}{\alpha n+\alpha+1+\alpha e^{\alpha nC_n}}\right)\vol(\mathbb{S}^n).
	\end{equation}
	Otherwise, if $k\ge n+2\ell+1$, then using Proposition \ref{thm6} we obtain
	\begin{equation*}
		n+2\ell+1\le k\le\frac{\vol(M)}{\vol(\mathbb{S}^n)}(e^t+n+1+nC_nt^{-1})-e^t,
	\end{equation*}
	which implies
	\begin{equation}\label{2intervel}
		\vol(M)\ge\left(1+\frac{2\alpha\ell-1}{\alpha n+\alpha+1+\alpha e^{\alpha nC_n}}\right)\vol(\mathbb{S}^n).
	\end{equation}
	Therefore, there exists a positive constant $B_{n,\alpha}<\alpha n+\alpha+1+\alpha\exp(\alpha nC_n)$ with $\alpha=1.43$ such that Theorem \ref{2ndd} is proved by \eqref{1intervel} and \eqref{2intervel}.
\end{proof}

\begin{rem}\label{compareremark}
	We compare the gap in Theorem \ref{2ndd} with the gap in Theorem \ref{1st} as following.
	\begin{itemize}
		\item[(i)] In the case $n+\ell+1\le k\le n+2\ell$, we have
			\begin{equation*}
				\begin{aligned}
					\left(\frac{\alpha\ell-1+\alpha(n+\ell+2)e^{E_{n,\ell}}}{B_{n,\alpha}}\right)\bigg/\left(\frac{\alpha\ell-1}{B_{n,\alpha}}\right)&=\frac{\alpha\ell-1+\alpha(n+\ell+2)e^{E_{n,\ell}}}{\alpha\ell-1}\\
					&=1+\frac{\alpha(n+\ell+2)e^{E_{n,\ell}}}{\alpha\ell-1}\\
					&>1+\frac{(n+\ell+2)e^{E_{n,\ell}}}{\ell}.
				\end{aligned}
			\end{equation*}
		\item[(ii)] In the case $k\ge n+2\ell+1$, we have
			\begin{equation*}
				\left(\frac{2\alpha\ell-1}{B_{n,\alpha}}\right)\bigg/\left(\frac{\alpha\ell-1}{B_{n,\alpha}}\right)=\frac{2\alpha\ell-1}{\alpha\ell-1}>\frac{2\alpha\ell}{\alpha\ell}=2.
			\end{equation*}
	\end{itemize}
	Hence, in both cases, the gaps in Theorem \ref{2ndd} is larger than the gap in Theorem \ref{1st}.
\end{rem}
Now we can give the proof of Theorem \ref{2nd}.
\begin{proof}[Proof of Theorem \ref{2nd}]
	Since $(n+2\ell)^{\frac{2}{n}},4^{\frac{1}{n}}>1$, we have
	\begin{equation*}
		\begin{aligned}
			E_{n,\ell}&= \alpha nC_n-\alpha nC_n(n+4)(n+2\ell)^{\frac{2}{n}}4^{\frac{1}{n}}\\
			&<\alpha nC_n-\alpha nC_n (n+4)\\
			&=-\alpha n(n+3)C_n.
		\end{aligned}
	\end{equation*}
	We claim that
	\begin{equation}\label{finaline}
		e^{-\alpha n(n+3)C_n}<\frac{\ell}{n+\ell+3}.
	\end{equation}
	In fact, if we fix $n$ and let $\ell$ vary, then the right-hand side tends to $1$ while the left-hand side is smaller than $0.1$. On the other hand, as $n$ tends to infinity, the left-hand side tends to $0$ much faster than the right-hand side. Hence, \eqref{finaline} is valid. Then we have
	\begin{equation*}
		(n+\ell+2)e^{E_{n,\ell}}<\ell,
	\end{equation*}
	which gives that $2\alpha\ell-1\ge\alpha\ell-1+\alpha(n+\ell+2)e^{E_{n,\ell}}$, then we have
	\begin{equation*}
		\vol(M)>\left(1+\frac{\alpha\ell-1+\alpha(n+\ell+2)e^{E_{n,\ell}}}{B_{n,\alpha}}\right)\vol(\mathbb{S}^n),
	\end{equation*}
	which is independent of the value of $k$. Also, Remark \ref{compareremark} (i) gives the proof of the second inequality.
\end{proof}


\end{document}